\documentclass{amsart}
\usepackage{amscd}
\usepackage{amsfonts}
\usepackage{amsmath}
\usepackage{amssymb}
\usepackage{amsthm}
\usepackage{fancyhdr}
\usepackage{latexsym}

\newtheorem{theorem}{Theorem}
\newtheorem{lemma}[theorem]{Lemma}

\newcommand{\R}{\mathbb R}

\begin{document}

\title[Ill-posedness for the HR equation]{Ill-posedness in the critical Sobolev space for the hyperelastic rod equation}

\author{Dan-Andrei Geba}

\address{Department of Mathematics, University of Rochester, Rochester, NY 14627, U.S.A.}
\email{dangeba@math.rochester.edu}
\date{}

\begin{abstract}
This article is the third in a series investigating ill-posedness in critical Sobolev spaces for Camassa-Holm-type equations. Following results for the $b$-Novikov and Fokas-Olver-Rosenau-Qiao equations, we study the hyperelastic rod equation at the critical Sobolev regularity $H^{3/2}(\mathbb R)$. We prove ill-posedness at this endpoint, complementing the local well-posedness theory in $H^s(\mathbb R)$ for $s>3/2$ and the recently established norm-inflation result for $1<s<3/2$.

The nature of the critical ill-posedness depends on the parameter $\gamma$. For $0<\gamma\leq3$, we construct arbitrarily small smooth initial data whose corresponding solutions develop finite-time blow-up, with the maximal lifespan tending to zero with the size of the initial data. For $\gamma<0$ or $\gamma>3$, we construct solutions with arbitrarily small $H^{3/2}$ initial data whose $H^{3/2}$ norm becomes arbitrarily large in arbitrarily short time. 
\end{abstract}


\maketitle

%%%%%%%%%%%%%%%%%%%%%%%%%%%%%%%%%%%%%%%%%%%%%%

\section{Introduction}

In this paper, we study the Cauchy problem for the non-periodic hyperelastic rod (HR) equation
\begin{equation}
\begin{cases}
u_t-u_{txx}+3uu_x=\gamma(2u_xu_{xx}+uu_{xxx}),\\
u(0,x)=u_0(x),\qquad x\in\mathbb{R},
\end{cases}
\label{cp}
\end{equation}
where $\gamma\neq0$. Our main result establishes norm inflation for \eqref{cp} in the critical Sobolev space $H^{3/2}(\mathbb R)$. The mechanism depends on the value of the parameter $\gamma$. For $0<\gamma\leq3$, arbitrarily small smooth initial data generate solutions that develop finite-time blow-up, with the maximal lifespan tending to zero with the size of the initial data. For $\gamma<0$ or $\gamma>3$, we construct solutions with arbitrarily small $H^{3/2}$ initial data whose $H^{3/2}$ norm becomes arbitrarily large in arbitrarily short time.

The HR equation was introduced by Dai \cite{D98} as a simplified model for finite-length, finite-amplitude waves propagating in cylindrical compressible hyperelastic rods. The parameter $\gamma$ plays an important role in the structure of the equation. In the particular case $\gamma=1$, \eqref{cp} becomes the Camassa-Holm equation \cite{CH93}, a fundamental model for shallow-water wave propagation. The Camassa-Holm equation is completely integrable and possesses peaked solitary waves, or peakons, on both the real line and the circle. At the other distinguished value $\gamma=0$, the HR equation reduces to the Benjamin-Bona-Mahony equation \cite{BBM72}, which describes unidirectional surface-wave propagation in a channel. Unlike the Camassa-Holm equation, the Benjamin-Bona-Mahony equation is not integrable, and its solitary waves are not solitons. Thus, the HR family connects equations with substantially different analytic and dynamical properties.

The well-posedness theory for \eqref{cp} has been studied extensively. Constantin and Strauss \cite{CS00} established local well-posedness (WP) in $H^s(\mathbb R)$ for $s>3/2$, investigated mechanisms leading to blow-up, and proved orbital stability of solitary waves. Yin \cite{Y03} subsequently extended the WP and blow-up analysis to a modified HR equation containing an additional $\omega u_x$ term. In the periodic setting, Zhou \cite{Z05} obtained analogues of the results of \cite{CS00}. These works identify $s=3/2$ as the natural regularity threshold for the standard energy method.

Further information about the behavior of the flow map above this threshold was obtained by Olson \cite{O06}, Karapetyan \cite{K10,K12}, and Inci \cite{I19}. Their results show that, for $s>3/2$, the data-to-solution map associated with the HR equation is not uniformly continuous in either $H^s(\mathbb R)$ or $H^s(\mathbb T)$. Such a loss of uniform continuity is characteristic of quasilinear evolution equations and further highlights the distinguished role of the index $s=3/2$.

The behavior below the critical index has recently been clarified by Li and Yu \cite{LY26}, who established norm inflation for the HR equation in $H^s(\mathbb R)$ for $1<s<3/2$. More precisely, they constructed smooth initial data that are arbitrarily small in $H^s$ but whose corresponding solutions develop arbitrarily large $H^s$ norm in arbitrarily short time. Their argument, which requires sufficient largeness of the spatial derivative on a nontrivial spatial region, also highlights an important distinction between the $1<s<3/2$ and $s=3/2$ regimes. At the $s=3/2$ critical index a large derivative at a single point can already be converted into divergence of the $H^{3/2}$ norm through a critical logarithmic estimate. Thus, the endpoint $s=3/2$ is not obtained simply by taking a limit in the argument for $1<s<3/2$ and requires a separate analysis.

The purpose of this paper is to establish norm inflation directly at the critical regularity $H^{3/2}(\mathbb R)$. We construct smooth initial data with arbitrarily small $H^{3/2}$ norm and exploit the critical Littlewood-Paley structure to produce an arbitrarily large slope at a single spatial point. For $0<\gamma\leq3$, an odd-symmetry blow-up argument then yields an explicit upper bound on the lifespan and shows that the corresponding $H^{3/2}$ norm becomes unbounded at the blow-up time. For $\gamma<0$ and $\gamma>3$, where no analogous quantitative blow-up estimate is available, we instead construct sequences of smooth initial data whose $H^{3/2}$ norms tend to zero while the corresponding maximal lifespans tend to zero. A critical $B^1_{\infty,\infty}$ continuation criterion then converts this behavior into norm inflation in $H^{3/2}$. 

An additional feature of the critical theory is the dependence of the ill-posedness (IP) mechanism on the parameter $\gamma$. For $0<\gamma\leq3$, arbitrarily small smooth initial data can generate solutions that lose boundedness in $H^{3/2}(\mathbb R)$ in arbitrarily short time, with the maximal lifespan itself becoming arbitrarily small. For $\gamma<0$ or $\gamma>3$, we instead obtain norm inflation in the usual sense: solutions emanating from data of arbitrarily small $H^{3/2}$ norm can reach arbitrarily large $H^{3/2}$ norm in arbitrarily short time. These two mechanisms lead to the following result.

\begin{theorem}
The Cauchy problem \eqref{cp} is IP in $H^{3/2}(\R)$. More specifically:

(i) Let $0<\gamma\leq 3$. Then, for every $\epsilon>0$, there exists initial data $u_0\in H^{\infty}(\R)$ with $\|u_0\|_{H^{3/2}(\R)}\leq \epsilon$ such that the Cauchy problem \eqref{cp} admits a unique solution 
\[
u\in C([0,T); H^\infty(\R))
\]
with maximal lifespan $T<\epsilon$ and
\begin{equation*}
\limsup_{t\to T-}\|u(t)\|_{H^{3/2}(\R)}=\infty.
\end{equation*}

(ii) Let either $\gamma<0$ or $\gamma>3$. Then, for every $\epsilon>0$, there exists a solution $u_\epsilon$ to \eqref{cp} and a time $0<t_\epsilon<\epsilon$ such that
\begin{equation}
\|u_\epsilon(0)\|_{H^{3/2}(\R)}\lesssim \epsilon
\label{0}
\end{equation}
and
\begin{equation}
\|u_\epsilon(t_\epsilon)\|_{H^{3/2}(\R)}\gtrsim \epsilon^{-1}.
\label{-1}
\end{equation}
\label{main-th}
\end{theorem}

Our critical result is closely related to the corresponding IP theory for the $b$-family of equations
\begin{equation}
u_t-u_{txx}+(b+1)u u_x=uu_{xxx}+b u_xu_{xx},
\label{b}
\end{equation}
which arises as a one-dimensional model for active fluid transport and includes the Camassa-Holm equation as the special case $b=2$. In an influential work, Guo, Liu, Molinet, and Yin \cite{GLM19} established norm inflation in $H^{3/2}$ for the $b$-family in the range $1<b\leq3$, both on the real line and on the circle, complementing the previously known local WP theory for $s>3/2$. More recently, Guo, Ye, and Yin \cite{G25} extended the corresponding non-periodic IP result to the range $b\neq1$. These works provide important precedents for the critical constructions used here. However, the HR equation has a different dependence on its parameter $\gamma$, and the endpoint behavior must be established directly within the HR dynamics. In particular, the arguments below adapt critical constructions from the $b$-family while exploiting the specific blow-up structure and conserved quantities of the HR equation.

We conclude the introduction by outlining the proof strategy. For the range $0<\gamma\leq3$, our argument is inspired by the critical construction of Guo et al. \cite{GLM19}. We construct a sequence of smooth initial data that remains arbitrarily small in $H^{3/2}$ while producing an arbitrarily large negative slope at the origin. Odd symmetry is preserved by the HR flow, and the resulting evolution of the slope at the origin satisfies a Riccati-type differential inequality. This yields an explicit upper bound for the maximal lifespan. To conclude that the $H^{3/2}$ norm itself must diverge at the blow-up time, we establish a critical continuation criterion based on the Besov norm $B^1_{\infty,\infty}$.

The cases $\gamma<0$ and $\gamma>3$ are more delicate, since the sign structure of the nonlocal term prevents us from applying the same quantitative blow-up criterion. In these regimes, we use the initial data introduced in \cite{G25}. The derivative at the origin then satisfies a Riccati-type differential inequality with a lower-order error term controlled by the conserved $H^1$ norm. A comparison argument shows that the corresponding maximal lifespans tend to zero. Combining this with the critical $B^1_{\infty,\infty}$ continuation criterion yields the asserted $H^{3/2}$ norm inflation. Although these arguments are motivated by the critical IP theory for the $b$-family, the endpoint results are proved here directly for the HR equation.

The paper is organized as follows. In Section 2, we collect the analytical tools needed in the proofs. We first recall the WP result of Constantin and Strauss for the HR equation and then establish several blow-up and continuation properties that parallel corresponding results for the $b$-family. In Section 3, we prove the finite-time blow-up assertion in part (i) of the main theorem. Section 4 is devoted to the norm-inflation result in part (ii).

Throughout the paper, we use the standard notation $A\lesssim B$ to denote $A\leq cB$ for some constant $c>0$. We write $A\sim B$ if both $A\lesssim B$ and $B\lesssim A$ hold. Finally, $A\ll B$ means that $A\leq cB$ for a sufficiently small constant $c>0$.

%%%%%%%%%%%%%%%%%%%%%%%%%%%%%%%%%%%%%%%%%%%%%%

\section{Toolbox}

We begin by recalling the relevant function spaces and WP results for the HR equation. Let $\mathcal{S}'(\R)$ denote the space of tempered distributions on $\R$, and define the Fourier transform of a function $f:\R\to\mathbb{C}$ by
\begin{equation*}
\widehat{f}(\xi):= \int_\R e^{-ix\xi} f(x)\,dx.
\end{equation*}
Let $\varphi: \R\to [0,1]$ be a smooth, radially decreasing function satisfying
\begin{equation*}
\varphi(\xi)=\begin{cases}
1, \quad |\xi|\leq 5/4,\\
0, \quad |\xi|\geq 8/5.
\end{cases}
\end{equation*}
For each integer $k\geq 0$, we define the Littlewood-Paley projectors $P_k$ by
\begin{equation*}
\widehat{P_0f}(\xi):=\varphi(\xi)\widehat{f}(\xi),
\end{equation*}
and, for $k\geq1$,
\begin{equation*}
\widehat{P_kf}(\xi)
:=
\left(
\varphi(\xi/2^k)-\varphi(\xi/2^{k-1})
\right)\widehat{f}(\xi).
\end{equation*}
We also define
\begin{equation*}
\widehat{P_{\leq k}f}(\xi):=\varphi(\xi/2^k) \widehat{f}(\xi), \quad \forall\,k\geq 1,
\end{equation*}
so that
\begin{equation*}
P_{\leq k}=\sum_{j=0}^k P_j,  \quad \forall\,k\geq 1.
\end{equation*}

The Besov space $B^s_{p,q}(\R)$ is defined by
\begin{equation*}
B^s_{p,q}(\R):=\left\{f\in \mathcal{S}'(\R);\ \|f\|_{B^s_{p,q}}:= \left\|2^{ks}\|P_k f\|_{L^p}\right\|_{l^q_k}<\infty\right\}.
\end{equation*}
We recall that 
\begin{equation*}
H^s(\R):=\left\{f\in \mathcal{S}'(\R);\ \|f\|_{H^s}:= \left(\int_\R (1+|\xi|^2)^s|\widehat{f}(\xi)|^2\,d\xi\right)^{1/2}<\infty\right\}.
\end{equation*}
and that $H^s(\R)$ coincides with $B^s_{2,2}(\R)$, with equivalent norms.

The main WP results for the HR equation established in \cite{CS00} are consolidated below.

\begin{theorem}
Let $s>3/2$ and $u_0\in H^s(\R)$. Then there exists a maximal time $T>0$ and a unique solution 
\begin{equation*}
u\in C([0,T); H^s(\R))\cap C^1([0,T); H^{s-1}(\R))
\end{equation*}
to the Cauchy problem \eqref{cp}, which depends continuously on the initial data $u_0$. If $T<\infty$, then
\begin{equation}
\limsup_{t\to T^-}\|u(t)\|_{H^s}=\infty.\label{hs-bl}
\end{equation}

Moreover, the quantities
\begin{equation*}
E(u)(t):=\int_{\mathbb R}
\left(u^2(x,t)+u_x^2(x,t)\right)\,dx
\end{equation*}
and
\begin{equation*}
F(u)(t):=\int_{\mathbb R}
\left(u^3(x,t)+\gamma u(x,t)u_x^2(x,t)\right)\,dx
\end{equation*}
are conserved in time.

Finally, the lifespan $T$ is independent of the Sobolev index $s$. In particular, if $u_0\in H^\infty(\R)$, then 
\[u\in C([0,T); H^\infty(\R)). 
\]
\label{wp-th} 
\end{theorem} 

%\begin{remark}
%Additionally, the maximal time $T$  satisfies the lower bound estimate
%\begin{equation*}
%T\geq \frac{c(s)}{\|u_0\|_{H^s}}
%\end{equation*}
%with $c(s)>0$ being a constant depending solely on $s$, as shown by Karapetyan \cite{K12}.
%\end{remark}

Next, we establish two finite-time blow-up criteria, each adapted to one of the two parts of our main result. The first applies to odd initial data and is similar to the one established in \cite{C98} for the Camassa-Holm equation.

\begin{lemma}
Assume that $0<\gamma\leq 3$. Let $u_0\in H^3(\mathbb R)$ be real-valued and odd, with $u_0'(0)<0$. Then the corresponding solution to \eqref{cp} blows up in finite time. Moreover, its lifespan $T$ satisfies
\[
T<-\frac{2}{\gamma u_0'(0)}.
\]
\label{bl-hr}
\end{lemma}

\begin{proof}
The HR equation is invariant under the transformation
\[
u(x,t)\mapsto -u(-x,t).
\]
By uniqueness, oddness of the initial datum therefore implies that the corresponding solution remains odd throughout its lifespan. In particular,
\begin{equation}
u(0,t)=u_{xx}(0,t)=0.
\label{u00}
\end{equation}

Let
\[
G(x):=\frac12 e^{-|x|}
\]
denote the Green's function of the operator $(1-\partial_x^2)^{-1}$ on $\mathbb R$. The HR equation can then be written in the equivalent form
\begin{equation}
u_t+\gamma uu_x
+\partial_x\!\left\{
G\ast\left(
\frac{3-\gamma}{2}u^2
+\frac{\gamma}{2}u_x^2
\right)
\right\}=0.\label{utg}
\end{equation}
Differentiating with respect to $x$ and rearranging the resulting terms, we obtain
\begin{equation}
u_{tx}+\gamma uu_{xx}+\frac{\gamma}{2}u_x^2
=
\frac{3-\gamma}{2}u^2
-
G\ast\left(
\frac{3-\gamma}{2}u^2
+\frac{\gamma}{2}u_x^2
\right).
\label{xhr}
\end{equation}

Evaluating \eqref{xhr} at $x=0$ and using \eqref{u00}, we find
\[
u_{tx}(0,t)
+\frac{\gamma}{2}u_x^2(0,t)
=
-
G\ast\left(
\frac{3-\gamma}{2}u^2
+\frac{\gamma}{2}u_x^2
\right)(0,t).
\]
Setting
\[
g(t):=u_x(0,t),
\]
and using $0<\gamma\leq3$, $G\geq0$, and $u^2,u_x^2\geq0$, we deduce
\[
g'(t)+\frac{\gamma}{2}g^2(t)\le 0,
\qquad t\in[0,T).
\]
Since $g(0)=u_0'(0)<0$, comparison with the solution of the Riccati equation
\[
y'+\frac{\gamma}{2}y^2=0,
\qquad
y(0)=u_0'(0),
\]
shows that $g(t)$ blows up to $-\infty$ no later than the blow-up time of $y$, namely
\[
-\frac{2}{\gamma u_0'(0)}.
\]
Consequently, the PDE solution cannot remain regular beyond this time, and hence
\[
T<-\frac{2}{\gamma u_0'(0)}.
\]
\end{proof}

The second blow-up criterion is more general and agrees with the corresponding criteria established for other Camassa-Holm-type equations; see, for example, \cite{BCD11} for the Camassa-Holm equation and \cite{G25} for the $b$-equation. We therefore only sketch the proof, following the standard argument used in these settings. In particular, our argument follows closely the proof given in \cite{BCD11}.

\begin{lemma}
Let $u_0$ be as in Theorem \ref{wp-th}, and let $u$ denote the corresponding solution. If the lifespan $T$ is finite, then
\begin{equation}
\int_0^T\,\|u(t)\|_{B^{1}_{\infty, \infty}}\,dt=\infty.\label{b1i}
\end{equation}
\end{lemma}

\begin{proof}
We rewrite the HR equation in transport form using \eqref{utg}:
\begin{equation*}
u_t+\gamma uu_x
=-\partial_x(1-\partial^2_x)^{-1}\left(
\frac{3-\gamma}{2}u^2
+\frac{\gamma}{2}u_x^2
\right).
\end{equation*}
Standard energy estimates for transport equations (e.g., Theorem 3.14 in \cite{BCD11}) yield
\begin{equation*}\aligned
&e^{-C\int_0^{t}\|u_x(t')\|_{L^\infty}\,dt'}\|u(t)\|_{H^s}\\ 
&\qquad \leq \|u_0\|_{H^s}
+C\int_0^t e^{-C\int_0^{t'}\|u_x(t'')\|_{L^\infty}\,dt''}\left(\|u^2(t')\|_{H^{s-1}}+\|u_x^2(t')\|_{H^{s-1}}\right)\,dt', \endaligned
\end{equation*}
for all $t\in [0,T)$. Here and below, $C>0$ denotes a constant depending only on the fixed parameters in the argument and may change from line to line. 

Since $s>1$, standard product estimates imply
\begin{equation*}
\|u^2(t')\|_{H^{s-1}}+
+\|u_x^2(t')\|_{H^{s-1}}\leq C (\|u(t')\|_{L^\infty}+\|u_x(t')\|_{L^\infty})\|u(t')\|_{H^s}.
\end{equation*}
Combining the preceding estimates and applying Gronwall's inequality, we obtain
\begin{equation*}
\|u(t)\|_{H^s}\leq  \|u_0\|_{H^s}\,e^{C\int_0^t \|u(t')\|_{L^\infty}+\|u_x(t')\|_{L^\infty}\,dt'}, \quad \forall \,t\in [0,T).
\end{equation*}

Since $s>3/2$, we can apply the logarithmic interpolation inequality (see (3.32) in \cite{BCD11})
\begin{equation}
\|w\|_{L^\infty}+\|w_x\|_{L^\infty}\leq C\left( 1+\|w\|_{B^1_{\infty,\infty}}\ln(1+\|w\|_{H^s})\right)\label{wx-li}
\end{equation}
and derive
\begin{equation*}
\aligned
\ln(1+\|u(t)\|_{H^s})\leq &\ln(1+ \|u_0\|_{H^s}) +Ct\\&+C\int_0^t \|u(t')\|_{B^1_{\infty,\infty}}\ln(1+\|u(t')\|_{H^s})\,dt', \quad \forall \,t\in [0,T).
\endaligned
\end{equation*}
A further application of Gronwall's inequality yields
\begin{equation*}
\ln(1+\|u(t)\|_{H^s})\leq \left(\ln(1+ \|u_0\|_{H^s}) +Ct\right)e^{C\int_0^t \|u(t')\|_{B^1_{\infty,\infty}}\,dt'}, \quad \forall \,t\in [0,T).
\end{equation*}

Suppose now, for contradiction, that 
\[
\int_0^T\,\|u(t)\|_{B^{1}_{\infty, \infty}}\,dt<\infty.
\]
Since $T<\infty$, the preceding estimate implies that
\[
\sup_{t\in[0,T)}\|u(t)\|_{H^s}<\infty.
\]
This contradicts the blow-up criterion \eqref{hs-bl}. Hence \eqref{b1i} follows.
\end{proof}

Finally, we record a logarithmic Gronwall inequality from \cite{GLM19}. 
\begin{lemma}
Let $A\in C^1([0,T);\R^+)$, with $0<T\leq \infty$, and suppose that 
\[
\frac{d}{dt} A(t)\leq B\, A(t)\ln(1+A(t)), \qquad \forall\,t\in [0,T),
\]
for some constant $B>0$. Then
\[
A(t)\leq (1+A(0))^{e^{Bt}}, \qquad \forall\,t\in [0,T).
\] 
\label{gron}
\end{lemma}

%%%%%%%%%%%%%%%%%%%%%%%%%%%%%%%%%%%%%%%%%%%%%%

\section{Proof of part (i) of Theorem \ref{main-th}}

We first recall that the HR equation can be equivalently written in momentum form as
\begin{equation}
y_t+\gamma y_x u +2\gamma yu_x+3(1-\gamma)uu_x=0, \qquad y=(1-\partial^2_x)u.\label{ne}
\end{equation}
For $u_0\in H^2$, the corresponding solution satisfies 
\[
u\in C([0,T); H^2(\R)).
\] 
Multiplying \eqref{ne} by $y$ and integrating by parts in $x$ gives
\begin{equation}
\aligned
\frac{d}{dt}\int_{\R}y^2\,dx&=-3\gamma \int_{\R} u_x\,y^2\,dx + 6(\gamma-1) \int_{\R} uu_x\,y\,dx\\
&\lesssim \|u_x\|_{L^\infty}\|y\|_{L^2}(\|u\|_{L^2}+\|y\|_{L^2}).
\endaligned
\label{y-en}
\end{equation}
Using the equivalence
\begin{equation}
\|y(t)\|_{L^2}=\|(1-\partial^2_x)u(t)\|_{L^2}\sim \|u(t)\|_{H^2},
\label{y-u}
\end{equation}
and the logarithmic interpolation inequality \eqref{wx-li}, we obtain
\begin{equation}
\aligned
\frac{d}{dt}\left\{\|y(t)\|^2_{L^2}\right\}&\lesssim  \|u_x\|_{L^\infty}\|y\|^2_{L^2}\\
&\lesssim 
\left(1+\|u(t)\|_{B^1_{\infty,\infty}}\ln(1+\|u(t)\|_{H^2})\right)\|y(t)\|^2_{L^2}\\
&\lesssim \left(1+\|u(t)\|_{B^1_{\infty,\infty}}\right)\ln(1+\|y(t)\|^2_{L^2})\|y(t)\|^2_{L^2}.
\endaligned
\label{diff-ineq}
\end{equation}

We now specify the initial data leading to finite-time blowup. This construction is based on the initial data introduced in \cite{GLM19} for the Camassa-Holm equation. For the cutoff function $\varphi$ introduced in the previous section, let
\[
\psi(\xi):=\varphi(\xi)-\varphi(2\xi).
\] 
Then
\[
\psi(\xi)=1, \qquad \forall\ \frac{4}{5}\leq |\xi|\leq \frac{5}{4}.
\]
Choose $\tilde{\varphi}\in C_0^\infty(\R)$ to be nonnegative and even, with 
\[
\text{supp}\,\tilde{\varphi}\subseteq\left\{\frac{4}{5}\leq |\xi|\leq \frac{5}{4}\right\}.
\]
It follows that
\[
\tilde{\varphi}\psi=\tilde{\varphi}.
\]
For $k\geq 1$, define $h_k\in \mathcal{S}(\R)$ by
\[
\widehat{h_k}(\xi):=i\frac{\xi}{2^k}\,\tilde{\varphi}(\frac{\xi}{2^k}),\qquad \forall\,k\geq 1,
\]
and set
\[
h:=\sum_{k\geq 1}\frac{1}{4^k\,k^{2/3}}\,h_k.
\]
Since $\tilde{\varphi}$ is real-valued and even, $h$ is real-valued and odd. Moreover, 
\begin{equation}
P_kh=\frac{1}{4^k\,k^{2/3}}\,h_k, \ \forall\,k\geq 1,\label{hk}
\end{equation}
and
\begin{equation}
 \|h\|_{H^{3/2}}\sim \left\|\frac{1}{k^{2/3}}\right\|_{l^2_k}<\infty.\label{h32}
\end{equation}
Fix $\epsilon>0$ and, for $N\geq 1$, define 
\begin{equation*}
u_{0,\epsilon}:=\frac{\epsilon}{\|h\|_{H^{3/2}}}\,P_{\leq N}h=\frac{\epsilon}{\|h\|_{H^{3/2}}}\,\sum_{1\leq k\leq N}\frac{1}{4^k\,k^{2/3}}\,h_k.
\end{equation*}
Then 
\[u_{0,\epsilon}\in \mathcal{S}(\R)\subset H^\infty(\R), \qquad 
\|u_{0,\epsilon}\|_{H^{3/2}}\leq \epsilon.
\] 
Furthermore, a direct computation gives
\[
-u'_{0,\epsilon}(0)= \frac{\epsilon}{2\pi\|h\|_{H^{3/2}}}\left(\int_\R \xi^2 \tilde{\varphi}(\xi)\,d\xi\right)\sum_{1\leq k\leq N}\,\frac{1}{k^{2/3}}.
\]
Since
\[
\sum_{k=1}^\infty \frac{1}{k^{2/3}}=\infty,
\]
we may choose $N$ sufficiently large, depending on $\epsilon$, $h$, and $\widetilde\varphi$, such that
\[
-u_{0,\epsilon}'(0)>\frac{2}{\gamma\epsilon}.
\]
By Theorem \ref{wp-th} and Lemma \ref{bl-hr}, the corresponding maximal solution satisfies
\[
u\in C([0,T); H^\infty(\R))
\] 
and
\[
T<-\frac{2}{\gamma u_{0,\epsilon}'(0)}<\epsilon.
\]

We claim that
\begin{equation}
\limsup_{t\to T-}\|u(t)\|_{B^1_{\infty,\infty}}=\infty.\label{b1ii}
\end{equation}
Suppose otherwise. Then there exists $M=M(\epsilon)>0$ such that
\[
\|u(t)\|_{B^1_{\infty,\infty}}\leq M, \qquad \forall\,t\in [0,T).
\] 
In view of \eqref{diff-ineq}, we have
\begin{equation}
\frac{d}{dt}\left\{\|y(t)\|^2_{L^2}\right\}\leq C(1+M)\ln(1+\|y(t)\|^2_{L^2})\|y(t)\|^2_{L^2}, \qquad \forall\,t\in [0,T).\label{diff-2}
\end{equation}
Lemma \ref{gron} implies that
\[
\sup_{t\in[0,T)}
\|y(t)\|_{L^2}
<\infty.
\]
By \eqref{y-u},
\[
\sup_{t\in[0,T)}
\|u(t)\|_{H^2}
<\infty.
\]
This contradicts \eqref{hs-bl}, since $T<\infty$ is the maximal lifespan of the solution. Hence \eqref{b1ii} holds.

Finally, the embedding
\[
B^{3/2}_{2,2}(\mathbb R)
\hookrightarrow
B^1_{\infty,\infty}(\mathbb R)
\]
gives
\begin{equation}
\|w\|_{B^1_{\infty,\infty}}
\lesssim
\|w\|_{B^{3/2}_{2,2}}
\sim
\|w\|_{H^{3/2}}.
\label{b1h32}
\end{equation}
Combining this estimate with \eqref{b1ii}, we conclude that
\[
\limsup_{t\to T-}\|u(t)\|_{H^{3/2}}=\infty.
\]
This completes the proof of part (i) of Theorem \ref{main-th}.

%%%%%%%%%%%%%%%%%%%%%%%%%%%%%%%%%%%%%%%%%%%%%%

\section{Proof of part (ii) of Theorem \ref{main-th}}

We argue by contradiction. Assume that there exists $\epsilon_0>0$ such that, for every initial datum satisfying
\[
\|u(0)\|_{H^{3/2}}\lesssim\epsilon_0
\]
and generating a maximal solution
\[
u\in C([0,T_{\rm max});H^{3/2}(\R))
\]
to the Cauchy problem \eqref{cp}, one has
\[
\|u\|_{L^\infty([0,\min\{\epsilon_0,T_{\rm max}\});H^{3/2})}
\lesssim\epsilon_0^{-1}.
\]
We derive a contradiction separately in the cases $\gamma<0$ and $\gamma>3$.

\subsection*{The case $\gamma<0$.}

We use the sequence of initial data introduced in \cite{G25} in the study of the $b$-equation. With the notation introduced in the previous section, define, for $N\geq2$,
\begin{equation*}
u_{0,N}
:=
-\frac1{\ln N}
\sum_{2\leq k\leq N}
\frac{1}{4^k k^{2/3}}h_k.
\end{equation*}
Then $u_{0,N}\in H^\infty(\mathbb R)$ is real-valued and odd. Moreover, \eqref{hk} and \eqref{h32} imply
\begin{equation}
\|u_{0,N}\|_{H^{3/2}}
\lesssim
\frac1{\ln N},
\label{n32}
\end{equation}
uniformly in $N$. A direct calculation gives
\[
u_{0,N}'(0)
=
\frac1{2\pi\ln N}
\left(
\int_{\mathbb R}\xi^2\widetilde\varphi(\xi)\,d\xi
\right)
\sum_{2\leq k\leq N}\frac1{k^{2/3}},
\]
and therefore
\begin{equation}
u_{0,N}'(0)
\gtrsim
\frac{N^{1/3}}{\ln N},
\label{u0n}
\end{equation}
also uniformly in $N$.

Let
\[
u_N\in C([0,T_{\max}^N);H^\infty(\mathbb R))
\]
be the corresponding maximal solution. Since $u_{0,N}$ is odd, uniqueness implies that
\[
u_N(x,t)=-u_N(-x,t)
\]
throughout its lifespan. Arguing as in the proof of Lemma \ref{bl-hr}, we obtain
\[
g_N'(t)+\frac{\gamma}{2}g_N^2(t)
=
-
G*
\left(
\frac{3-\gamma}{2}u_N^2
+
\frac{\gamma}{2}(u_N)_x^2
\right)(0,t),\qquad
t\in[0,T_{\rm max}^N),
\]
where
\[
g_N(t):=(u_N)_x(0,t).
\]
Since $\gamma<0$, we may discard the nonnegative contribution coming from $(u_N)_x^2$ on the right-hand side and obtain
\[
g_N'(t)+\frac{\gamma}{2}g_N^2(t)
\geq
-
G*
\left(
\frac{3-\gamma}{2}u_N^2
\right)(0,t), \qquad
t\in[0,T_{\rm max}^N).
\]

By Young's inequality,
\[
\begin{aligned}
G*\!\left(\frac{3-\gamma}{2}u_N^2\right)(0,t)
&\lesssim
\|G\|_{L^\infty}\|u_N^2(t)\|_{L^1}\\
&\lesssim
\|u_N(t)\|_{L^2}^2.
\end{aligned}
\]
Since the $H^1$ norm is conserved,
\[
\|u_N(t)\|_{L^2}^2
\le
\|u_N(t)\|_{H^1}^2
=
\|u_{0,N}\|_{H^1}^2
\le
\|u_{0,N}\|_{H^{3/2}}^2
\lesssim
\frac1{\ln^2N}.
\]
Consequently,
\[
g_N'(t)
\ge
\frac{-\gamma}{2}g_N^2(t)
-
\frac{C}{\ln^2N},
\qquad
t\in[0,T_{\rm max}^N),
\]
for some constant $C>0$ independent of $N$.

By \eqref{u0n},
\[
g_N(0)
\gtrsim
\frac{N^{1/3}}{\ln N},
\]
and hence, for $N$ sufficiently large,
\[
\frac{-\gamma}{2}g_N^2(0)-\frac{C}{\ln^2N}>0.
\]
Consider the solution $z_N$ of the Riccati equation
\[
z_N'=\frac{-\gamma}{2}z_N^2
-
\frac{C}{\ln^2N},
\qquad
z_N(0)=g_N(0).
\]
By the comparison principle,
\[
g_N(t)\geq z_N(t)
\]
as long as both functions exist. Since 
\[
\frac{-\gamma}{2}z_N^2(0)-\frac{C}{\ln^2N}=\frac{-\gamma}{2}g_N^2(0)-\frac{C}{\ln^2N}>0,
\]
$z_N$ blows up to $+\infty$ in finite time. Its blow-up time is
\[
\tau_N
=
\frac{\ln N}{\sqrt{-2C\gamma}}
\ln
\left(
\frac{
\sqrt{-\gamma/2}\,u_{0,N}'(0)+C^{1/2}/\ln N
}{
\sqrt{-\gamma/2}\,u_{0,N}'(0)-C^{1/2}/\ln N
}
\right).
\]
Using
\[
\ln\left(\frac{1+x}{1-x}\right)
\leq
\frac{2x}{1-x},
\qquad 0<x<1,
\]
we obtain
\[
\tau_N
\lesssim
\frac{1}{
\sqrt{-\gamma/2}\,u_{0,N}'(0)-C^{1/2}/\ln N}.
\]
In view of \eqref{u0n},
\[
\lim_{N\to \infty}\tau_N=0,
\]
and, since $T_{\max}^N$ cannot exceed the blow-up time of the comparison solution, we have
\[
\lim_{N\to \infty}T_{\max}^N=0.
\]
Thus, for $N$ sufficiently large,
\[
T_{\max}^N<\epsilon_0.
\]

At the same time, \eqref{n32} implies that
\[
\|u_{0,N}\|_{H^{3/2}}
\lesssim
\epsilon_0
\]
for all sufficiently large $N$. The contradiction hypothesis therefore gives
\[
\|u_N\|_{L^\infty([0,T_{\max}^N);H^{3/2})}
\lesssim
\epsilon_0^{-1}.
\]
Using \eqref{b1h32},
\[
\|u_N\|_{L^\infty([0,T_{\max}^N);B^1_{\infty,\infty})}
\lesssim
\epsilon_0^{-1}.
\]
Hence
\[
\int_0^{T_{\max}^N}
\|u_N(t)\|_{B^1_{\infty,\infty}}\,dt
<\infty.
\]
This contradicts \eqref{b1i}, since $T_{\max}^N<\infty$. Therefore the contradiction hypothesis is false when $\gamma<0$.

\subsection*{The case $\gamma>3$.}

We now consider the sequence of initial data
\[
v_{0,N}:=-u_{0,N},
\qquad N\geq2.
\]
Then $v_{0,N}\in H^\infty(\mathbb R)$ is real-valued and odd and
\[
\|v_{0,N}\|_{H^{3/2}}
=
\|u_{0,N}\|_{H^{3/2}}
\lesssim
\frac1{\ln N},
\]
while
\begin{equation}
-v_{0,N}'(0)
=
u_{0,N}'(0)
\gtrsim
\frac{N^{1/3}}{\ln N}.
\label{v0n}
\end{equation}

Let
\[
v_N\in C([0,T_{\max}^N);H^\infty(\mathbb R))
\]
be the corresponding maximal solution, which, by uniqueness, remains odd throughout its lifespan. Hence, we can argue as in the proof of Lemma \ref{bl-hr} and derive
\[
h_N'(t)+\frac{\gamma}{2}h_N^2(t)
=
-
G*
\left(
\frac{3-\gamma}{2}v_N^2
+
\frac{\gamma}{2}(v_N)_x^2
\right)(0,t), 
\qquad
t\in[0,T_{\rm max}^N),
\]
where
\[
h_N(t):=(v_N)_x(0,t).
\]
Since $\gamma>3$, this gives
\[
h_N'(t)+\frac{\gamma}{2}h_N^2(t)
\leq
G*
\left(
\frac{\gamma-3}{2}v_N^2
\right)(0,t),
\qquad
t\in[0,T_{\rm max}^N).
\]

As in the previous case, conservation of the $H^1$ norm and Young's inequality imply
\[
G*
\left(
\frac{\gamma-3}{2}v_N^2
\right)(0,t)
\lesssim
\|v_N(t)\|_{L^2}^2
\lesssim
\frac1{\ln^2N}.
\]
Thus
\[
h_N'(t)
\leq
-\frac{\gamma}{2}h_N^2(t)
+
\frac{C}{\ln^2N},
\qquad
t\in[0,T_{\rm max}^N),
\]
for some constant $C>0$ independent of $N$.

By \eqref{v0n},
\[
h_N(0)
=
v_{0,N}'(0)
\lesssim
-\frac{N^{1/3}}{\ln N},
\]
and in particular
\[
-\frac{\gamma}{2}h_N^2(0)
+
\frac{C}{\ln^2N}
<0
\]
for $N$ sufficiently large. If we let $z_N$ be the solution of the Riccati equation
\[
z_N'=-\frac{\gamma}{2}z_N^2(t)
+
\frac{C}{\ln^2N},
\qquad
z_N(0)=h_N(0),
\]
then $z_N$ blows up to $-\infty$ in finite time. By comparison,
\[
h_N(t)\leq z_N(t)
\]
up to the blow-up time. Arguing as in the previous case, we obtain
\[
\lim_{N\to \infty}T_{\max}^N=0.
\]
In particular, after increasing $N$ if necessary,
\[
T_{\max}^N<\epsilon_0.
\]

On the other hand,
\[
\|v_{0,N}\|_{H^{3/2}}
=
\|u_{0,N}\|_{H^{3/2}}
\lesssim
\frac1{\ln N},
\]
so for sufficiently large $N$,
\[
\|v_{0,N}\|_{H^{3/2}}
\lesssim
\epsilon_0.
\]
The contradiction hypothesis therefore yields
\[
\|v_N\|_{L^\infty([0,T_{\rm max}^N);H^{3/2})}
\lesssim
\epsilon_0^{-1}.
\]
By \eqref{b1h32},
\[
\|v_N\|_{L^\infty([0,T_{\rm max}^N);B^1_{\infty,\infty})}
\lesssim
\epsilon_0^{-1}.
\]
Consequently,
\[
\int_0^{T_{\max}^N}
\|v_N(t)\|_{B^1_{\infty,\infty}}\,dt
<\infty,
\]
contradicting \eqref{b1i} because $T_{\max}^N<\infty$. This contradiction completes the proof of part (ii) of Theorem \ref{main-th} in the case $\gamma>3$.

%%%%%%%%%%%%%%%%%%%%%%%%%%%%%%%%%%%%%%%%%%%%%%

\section*{Acknowledgements}
The author is deeply grateful to God and the Most Holy Theotokos for Their love, patience, and mercy, which helped him while working on this paper. 

He is also thankful to Alex Himonas for his friendship and for introducing and guiding him to work on Camassa-Holm-type equations. 

%%%%%%%%%%%%%%%%%%%%%%%%%%%%%%%%%%%%%%%%%%%%%%

\bibliographystyle{amsplain}
\bibliography{bNE}

@article {CH93,
    AUTHOR = {Camassa, R. and Holm, D.},
     TITLE = {An integrable shallow water equation with peaked solitons},
   JOURNAL = {Phys. Rev. Lett.},
  FJOURNAL = {Physical Review Letters},
    VOLUME = {71},
      YEAR = {1993},
    NUMBER = {11},
     PAGES = {1661--1664},
}

@article {GLM19,
    AUTHOR = {Guo, Z. and Liu, X. and Molinet, L. and Yin,
              Z.},
     TITLE = {Ill-posedness of the {C}amassa-{H}olm and related equations in
              the critical space},
   JOURNAL = {J. Differential Equations},
  FJOURNAL = {Journal of Differential Equations},
    VOLUME = {266},
      YEAR = {2019},
    NUMBER = {2-3},
     PAGES = {1698--1707},
}

@article {D98,
    AUTHOR = {Dai, H.-H.},
     TITLE = {Model equations for nonlinear dispersive waves in a
              compressible {M}ooney-{R}ivlin rod},
   JOURNAL = {Acta Mech.},
  FJOURNAL = {Acta Mechanica},
    VOLUME = {127},
      YEAR = {1998},
    NUMBER = {1-4},
     PAGES = {193--207},
}

@article {BBM72,
    AUTHOR = {Benjamin, T. B. and Bona, J. L. and Mahony, J. J.},
     TITLE = {Model equations for long waves in nonlinear dispersive
              systems},
   JOURNAL = {Philos. Trans. Roy. Soc. London Ser. A},
  FJOURNAL = {Philosophical Transactions of the Royal Society of London.
              Series A. Mathematical and Physical Sciences},
    VOLUME = {272},
      YEAR = {1972},
    NUMBER = {1220},
     PAGES = {47--78},
}

@article {CS00,
    AUTHOR = {Constantin, A. and Strauss, W. A.},
     TITLE = {Stability of a class of solitary waves in compressible elastic
              rods},
   JOURNAL = {Phys. Lett. A},
  FJOURNAL = {Physics Letters. A},
    VOLUME = {270},
      YEAR = {2000},
    NUMBER = {3-4},
     PAGES = {140--148},
}

@article {Y03,
    AUTHOR = {Yin, Z.},
     TITLE = {On the {C}auchy problem for a nonlinearly dispersive wave
              equation},
   JOURNAL = {J. Nonlinear Math. Phys.},
  FJOURNAL = {Journal of Nonlinear Mathematical Physics},
    VOLUME = {10},
      YEAR = {2003},
    NUMBER = {1},
     PAGES = {10--15},
}

@article {Z05,
    AUTHOR = {Zhou, Y.},
     TITLE = {Local well-posedness and blow-up criteria of solutions for a
              rod equation},
   JOURNAL = {Math. Nachr.},
  FJOURNAL = {Mathematische Nachrichten},
    VOLUME = {278},
      YEAR = {2005},
    NUMBER = {14},
     PAGES = {1726--1739},}

@article {I19,
    AUTHOR = {Inci, H.},
     TITLE = {On the well-posedness of the hyperelastic rod equation},
   JOURNAL = {Ann. Mat. Pura Appl. (4)},
  FJOURNAL = {Annali di Matematica Pura ed Applicata. Series IV},
    VOLUME = {198},
      YEAR = {2019},
    NUMBER = {3},
     PAGES = {795--802},
}

@book {K12,
    AUTHOR = {Karapetyan, D.},
     TITLE = {On the well-poseness of the hyperelastic rod equation},
      NOTE = {Thesis (Ph.D.)--University of Notre Dame},
 PUBLISHER = {ProQuest LLC, Ann Arbor, MI},
      YEAR = {2012},
     PAGES = {103},
}

@article {K10,
    AUTHOR = {Karapetyan, D.},
     TITLE = {Non-uniform dependence and well-posedness for the hyperelastic
              rod equation},
   JOURNAL = {J. Differential Equations},
  FJOURNAL = {Journal of Differential Equations},
    VOLUME = {249},
      YEAR = {2010},
    NUMBER = {4},
     PAGES = {796--826},
}

@book {BCD11,
    AUTHOR = {Bahouri, H and Chemin, J.-Y. and Danchin, R.},
     TITLE = {Fourier analysis and nonlinear partial differential equations},
    SERIES = {Grundlehren der mathematischen Wissenschaften [Fundamental
              Principles of Mathematical Sciences]},
    VOLUME = {343},
 PUBLISHER = {Springer, Heidelberg},
      YEAR = {2011},
     PAGES = {xvi+523},
}

@article {G25,
    AUTHOR = {Guo, Y. and Ye, W. and Yin, Z.},
     TITLE = {Sharp ill-posedness and global well-posedness for the
              {$b$}-equation},
   JOURNAL = {J. Differential Equations},
  FJOURNAL = {Journal of Differential Equations},
    VOLUME = {434},
      YEAR = {2025},
     PAGES = {Paper No. 113297, 12},}

@article {C98,
    AUTHOR = {Constantin, A. and Escher, J.},
     TITLE = {Global existence and blow-up for a shallow water equation},
   JOURNAL = {Ann. Scuola Norm. Sup. Pisa Cl. Sci. (4)},
  FJOURNAL = {Annali della Scuola Normale Superiore di Pisa. Classe di
              Scienze. Serie IV},
    VOLUME = {26},
      YEAR = {1998},
    NUMBER = {2},
     PAGES = {303--328},}

@article {O06,
    AUTHOR = {Olson, E.},
     TITLE = {Non-uniform dependence on initial data for a family of non-linear evolution equations},
   JOURNAL = {Differential Integral Equations},
  FJOURNAL = {Differential and Integral Equations. An International Journal
              for Theory \& Applications},
    VOLUME = {19},
      YEAR = {2006},
    NUMBER = {10},
     PAGES = {1081--1102},
}

@unpublished {LY26,
 AUTHOR = {Li, J. and Yu, Y.},
TITLE = { Norm inflation and low-regularity ill-posedness for the rod equation},
NOTE = {preprint, arXiv:2604.12423},
}
  
\end{document}